\documentclass[a4paper,12pt]{amsart}

\DeclareRobustCommand{\SkipTocEntry}[5]{}

\usepackage{shadow}

\usepackage{accents}
\usepackage{marvosym}

\usepackage[T1]{fontenc}
\usepackage{empheq}
\usepackage{relsize}
\usepackage{amsmath}

\usepackage{bm}

\usepackage{upgreek}
\usepackage{ esint }
\usepackage{color}
\usepackage{comment}
\usepackage{amssymb}
\usepackage{amsfonts}
\usepackage{graphicx}

\usepackage{amscd}

\usepackage{slashed}
\usepackage{pdflscape}
\usepackage{tikz}
\usepackage{enumerate}
\usepackage{multirow}
\usepackage{stmaryrd}
\usepackage{cancel}
\usepackage{t1enc}

\usepackage{amssymb}
\usepackage{amscd}
\usepackage{scalerel,stackengine}

\usepackage{ifthen}

\usepackage{pdfsync}

\usepackage{graphicx} 
\usepackage{amsmath} 
\usepackage{amsfonts}
\usepackage{amssymb}

\usepackage{bbold}
\usepackage[linkcolor=blue]{hyperref}
\usepackage{mathrsfs}
\usepackage[colorinlistoftodos]{todonotes}
\usepackage{ytableau}

\definecolor{blue}{rgb}{.255,.41,.884} 
\definecolor{red}{rgb}{1, 0, 0} 
\definecolor{green}{rgb}{.196,.804,.196} 
\definecolor{yellow}{rgb}{1,.648,0} 
\definecolor{pink}{rgb}{1,0.5,0.5}

\newtheorem{theorem}{Theorem}[section]
\newtheorem{lemma}[theorem]{Lemma}

\theoremstyle{definition}
\newtheorem{definition}[theorem]{Definition}

\theoremstyle{remark}
\newtheorem{remark}[theorem]{Remark}

\newcommand{\IIo}{\hspace{0.4mm}\mathring{\rm{ I\hspace{-.2mm} I}}{\hspace{.0mm}}}

\newcommand{\IIIo}{{\mathring{{\bf\rm I\hspace{-.2mm} I \hspace{-.2mm} I}}{\hspace{.2mm}}}{}}
\newcommand{\IVo}{{\mathring{{\bf\rm I\hspace{-.2mm} V}}{\hspace{.2mm}}}{}}

\newcommand{\otop}{\mathring{\top}}
\newcommand{\ID}{{{{\bf\rm I\hspace{-.3mm} D}}{\hspace{0.1mm}}}}

\newcounter{mnotecount}

\newcommand{\mnotex}[1]
{\protect{\stepcounter{mnotecount}}$^{\mbox{\footnotesize $\bullet$\themnotecount}}$ 
\marginpar{
\raggedright\tiny\em
$\!\!\!\!\!\!\,\bullet$\themnotecount: #1} }

\numberwithin{equation}{section}

\newcommand{\cc}{\boldsymbol{c}}

\renewcommand{\=}{\stackrel\Sigma =}

\renewcommand\geq{\geqslant}

\stackMath
\newcommand\reallywidehat[1]{%
\savestack{\tmpbox}{\stretchto{%
  \scaleto{%
    \scalerel*[\widthof{\ensuremath{#1}}]{\kern-.6pt\bigwedge\kern-.6pt}%
    {\rule[-\textheight/2]{1ex}{\textheight}}
  }{\textheight}%
}{0.5ex}}%
\stackon[1pt]{#1}{\tmpbox}%
}

\newcommand{\bdot }{\mathop{\lower0.33ex\hbox{\LARGE$\cdot$}}}

\definecolor{ao}{rgb}{0.0,0.0,1.0}
\definecolor{forest}{rgb}{0.0,0.3,0.0}
\definecolor{red}{rgb}{0.8, 0.0, 0.0}

\newcommand{\FF}[1]{\mathring{\underline{\overline{\rm{#1}}}}}

\begin{document}

\subjclass[2010]{
53C18, 53A55, 53C21, 58J32.
}

\renewcommand{\today}{}
\title{Non-existence of higher-order conformal fundamental forms in odd dimensions
}
%
%
%

\author{ Samuel Blitz${}^\flat$ }

\address{${}^\flat$
 Department of Mathematics and Statistics \\
 Masaryk University\\
 Building 08, Kotl\'a\v{r}sk\'a 2 \\
 Brno, CZ 61137} 
   \email{blitz@math.muni.cz}
 
\vspace{10pt}


\maketitle

\pagestyle{myheadings} \markboth{S. Blitz}{S. Blitz}

\begin{abstract}
Conformal fundamental forms populate a minimal generating set for low differential order invariants of conformal hypersurface embeddings. In this work we complete the characterization of conformal fundamental forms by proving the general non-existence of higher-order conformal fundamental forms when the embedded hypersurface is even dimensional.
\vspace{0.5cm}

\noindent
{\sf \tiny Keywords: 
Conformal geometry, extrinsic conformal geometry and hypersurface embeddings, conformally compact}

\end{abstract}

\vspace{2cm}

\section{Introduction}

The Poincar\'e--Einstein problem describes the problem of finding an Einstein metric in the interior of conformally compact manifold. After conformal compactification, the problem may be recast as solving the almost Einstein equation~\cite{GoverAE} on the conformal manifold $M^d$ with boundary $\Sigma$, with boundary data fixed by the embedding of the conformal boundary. It was observed in~\cite{LeBrun,FGbook,GoverAE} that any solution necessarily has an umbilic boundary, making the trace-free second fundamental form an obstruction to the existence of a solution. Later work by~\cite{curry-G} established the presence of another obstruction: the Fialkow tensor~\cite{Fialkow}. Notably, both the trace-free second fundamental form and the Fialkow tensor are conformally invariant. Recently~\cite{BGW1}, a finite family of conformally invariant tensors that are higher jet analogs of the trace-free second fundamental form and the Fialkow tensor, called the \textit{conformal fundamental forms}, were constructed that provide an essentially complete accounting of the (formal) obstructions to solutions of the almost Einstein equation. Roughly, a $k$th conformal fundamental form denoted $\FF{k}$ is a conformally-invariant section of $\odot^2_{\circ} T^* \Sigma$ that has as its leading term $k-1$ derivatives of the metric in the transverse direction---see Section~\ref{formalism} for a more precise definition. The trace-free second fundamental form $\IIo$ and the trace-free part of the Fialkow tensor $\IIIo$ are the first two elements of this family.

On conformal manifolds more broadly, the problem of classifying invariants of conformal submanifold embeddings is of recent interest~\cite{burstall-calderbank,mondino2018,curry-gover-snell,blitz-silhan}. Most interest is in constructing integrated conformal invariants, but by analogy with the seminal work of Alexakis, it is conjectured (and indeed proven for low-lying dimensions) that integrated conformal invariants of 2- and 4-dimensional submanifolds can be expressed in terms of a topological invariant, integrals along the submanifold of conformally-invariant scalars, and integrals of divergences~\cite{CGKTW}. The classification problem, then, may be decomposed to include the problem of classifying conformal invariants of submanifolds.

The problem in codimension larger-than-one is still quite open (see~\cite{curry-gover-snell,blitz-silhan} for examples of recent efforts), but significant work has gone toward understanding conformal hypersurface embeddings (see ~\cite{graham-volume,mollers-orsted-zhang,case2018,GPt,will1,blitz-classifying} for a small sample). Indeed, in~\cite{blitz-classifying} it was shown that when there exists conformal fundamental forms $\{\IIo, \IIIo, \ldots, \FF{k}\}$, they necessarily appear in a minimal generating set for conformal hypersurface invariants that include at most $k-1$ normal derivatives of the metric. Given this result, it is important to establish precisely when these conformal fundamental forms exist; in~\cite{BGW1}, it was shown that for each $k \in \{2, \ldots, \lceil \tfrac{d+1}{2} \rceil \}$, a conformal fundamental form $\FF{k}$ exists, and in~\cite{Blitz1} it was shown that when $d$ is even, there exist conformal fundamental forms $\FF{k}$ for each $k \in \{2, \ldots, d-1\}$. As a consequence of the nonlinear Dirichlet--to--Neumann map for the almost Einstein problem being of order $d-1$, we do not consider conformal fundamental forms at this ``critical'' transverse order (or beyond).

In this paper, we provide the classification of ``subcritical'' conformal fundamental forms, proving that the constructions described above are complete. This is captured in the main result:
\begin{theorem} \label{main-theorem}
For $d \geq 4$ even and every $k \in \{2, \ldots, d-1\}$ there exists a conformal fundamental form $\FF{k}$. For $d \geq 5$ odd, for each $k \in \{2, \ldots, \tfrac{d+1}{2}\}$ there exists a conformal fundamental form $\FF{k}$, but for each $k \in \{\tfrac{d+3}{2}, \ldots, d-1\}$ there does not exist a conformal fundamental form $\FF{k}$.
\end{theorem}
As the existence of the aforementioned conformal fundamental forms was proven elsewhere, proving this theorem amounts to proving nonexistence of certain higher-order conformal fundamental forms when $d$ is odd. Briefly, the theorem is proven in two steps. First, we establish that certain conformally invariant differential operators mapping sections in $M$ to sections in $\Sigma$ do not exist; second, we establish that the existence of the higher order conformal fundamental forms in odd dimensions would imply the existence of such operators. This proof runs parallel to the proof of the existence of GJMS-like operators on $\Sigma$~\cite{Matsumoto}, and the failure of these conformal fundamental forms to exist is similar in spirit to the failure of higher-order GJMS operators to exist in low-lying dimensions~\cite{GoverHirachi}.

\section{Conformal Hypersurface Operators and Fundamental Forms} \label{formalism}


We begin by providing careful definitions of natural conformal hypersurface operators and invariants. To that end, we first define
$$\mathcal{M}^d := \{(\Sigma, (M,\cc), \iota) \;|\; \iota: \Sigma \hookrightarrow (M^d,\cc)\}\,,$$
the space of conformal hypersurface embeddings into $d$-dimensional conformal manifolds. Going forward, all manifolds $\Sigma$ and $M$ are assumed to be smooth manifolds, and $\cc$ is a conformal structure, where Riemannian metrics $g,g' \in \cc$ implies that there exists a positive function $\Omega \in C^\infty M$ such that $g' = \Omega^2 g$. Now given an embedding $\iota : \Sigma \hookrightarrow (M^d,\cc)$, we may pull back the conformal structure via the embedding, giving rise to a unique conformal structure $\bar{\cc} := \iota^* \cc$ on the hypersurface $\Sigma$. Note that we will use overbars $\bar{\bullet}$ to denote geometric quantities and operators intrinsic to manifold $(\Sigma,\bar{\cc})$.

On conformal manifolds, the natural geometric objects are sections of density bundles (see~\cite{} for an explicit construction). For our purposes, we may view these sections as double equivalence classes of metric representatives and sections of a vector bundle over $M$ (or $\Sigma$), with equivalence defined by mutual rescaling of the metric representative and the section. For example, we can define the weight-$w$ section $\eta$ according to
$$\eta := [g ; h] = [\Omega^2 g; \Omega^w h] \in \Gamma(\odot^2_{\circ} T^* M[w])\,.$$


To any element $(\Sigma,(M,\cc),\iota) \in \mathcal{M}^d$, we may associate a family of natural conformally-invariant differential operators (see~\cite{GPt} for a precise definition). That is, we define
\begin{align*}
\operatorname{Ops}^{(d)}_{(\Sigma,(M,\cc),\iota)} := \{&op :  \Gamma(T^\Phi M[w']) \rightarrow \Gamma(T^\Psi \Sigma[w']) \; | \\& \Phi, \Psi \text{ irreducible representations on } TM, T\Sigma, \text{ respectively}\; | \\& op \text{ natural}\}\,.
\end{align*}
By ``natural'' in the above definition, we merely mean an extension to hypersurface embeddings of the typical naturalness notion used in Riemannian geometry~\cite{stredder}---this is formally provided in~\cite{GPt}. For practical purposes, this amounts to an operator $op \in \operatorname{Ops}^{(d)}_{(\Sigma,(M,\cc),\iota)}$ being expressible (after trivializing the conformal structure to a metric representative $g \in \cc$) as a polynomial in the Levi-Civita connection $\nabla$ of $(M,g)$ with tensor-valued pre-invariants as defined in~\cite[Section 2.4]{will1}. In fact, these coefficients may equivalently each be expressed as a polynomial in the metric $g$, its inverse $g^{-1}$, the unit conormal $\hat{n}$, the Riemann curvature tensor $R$, and their covariant derivatives. As each space of natural conformal hypersurface differential operators corresponds to a single element of $\mathcal{M}^d$, we may define the total space
$$\operatorname{Ops}^{(d)} := \bigsqcup_{m \in \mathcal{M}^d} \operatorname{Ops}^{(d)}_m\,,$$
equipped with a natural projection $\pi := \operatorname{Ops}^{(d)} \rightarrow \mathcal{M}^d$. As such, a natural conformal hypersurface differential operator corresponds to any section
$$op : \mathcal{M}^d \rightarrow \operatorname{Ops}^{(d)}$$
such that $\pi \circ op = \operatorname{id}_{\mathcal{M}^d}$. (Note that we may similarly define ``bulk'' conformal operators from $\Gamma(T^\Phi M[w])$ to $\Gamma(T^\Psi M[w'])$.)

\begin{remark}
Observe that one may recast the above construction in the language of category theory, but for our purposes it is not required.
\end{remark}

We may also define natural conformal hypersurface invariants. Indeed, a natural conformal hypersurface invariant is identified with any zeroth order natural conformal hypersurface differential operator with domain given by the trivial representation. Viewing these invariants as the images of those differential operators on the unit function, we can instead treat them as sections
\begin{align*}
I : \mathcal{M}^d \rightarrow \bigsqcup_{m \in \mathcal{M}^d} \{& \mathcal{I} \in \Gamma(T^\Psi \Sigma_m[w']) \;| \\& \Psi \, \text{ an irreducible representation on } T \Sigma_m \; | \; \mathcal{I} \text{ natural}\}\,,
\end{align*}
where $m = (\Sigma_m, (M^d,\cc),\iota)$ and $\varpi \circ I = \operatorname{id}_{\mathcal{M}^d}$, with $\varpi$ the natural projection to $\mathcal{M}^d$.

For the purposes of this article, we require a specific family of natural conformal hypersurface invariants, occupying the total space of sections of symmetric trace-free rank 2 weight $3-k$ tensors on hypersurfaces:
$$E_k := \bigsqcup_{m \in \mathcal{M}^d} \{e_k \in \Gamma(\odot^2_{\circ} T^* \Sigma_m[3-k]) \; | \; e_k \text{ natural}\}\,.$$
As before, this space is equipped with the natural projection $\pi_k : E_k \rightarrow \mathcal{M}^d$. We can now define conformal fundamental forms.

\begin{definition}
A $k$th conformal fundamental form $\FF{k}$ is any section
$$\FF{k} : \mathcal{M}^d \rightarrow E_k\,,$$
satisfying:
\begin{itemize}
\item $\pi_k \circ \FF{k} = \operatorname{id}_{\mathcal{M}^d}$,
\item for any $(\Sigma,(M^d,\cc),\iota) \in \mathcal{M}^d$, a trivialization $g \in \cc$, and coordinates $(s, x^1, \ldots, x^{d-1})$ in a small neighborhood around a point $p \in \iota(\Sigma)$ with $s$ a unit defining function for $\Sigma$, we may express
$$\FF{k}(\Sigma,(M^d,\cc),\iota) = \otop \partial_s^{k-1} g_{ab} + \text{ltots}(g)\,.$$
\end{itemize}
\end{definition}
In the above definition, we denote $\otop := \operatorname{tf} \circ \top$, where $\operatorname{tf}$ indicates the trace-free projection of the tensor and $\top$ denotes the projection operator to the hypersurface. A unit defining function is any function $s \in C^\infty M$ such that $s \= 0$ and $|ds|_g = 1$, and $\operatorname{ltots}$ is any linear differential operator that takes strictly fewer than $k-1$ partial derivatives $\partial_s$.

There exist known~\cite{BGW1} symbolic formulas for some low-lying conformal fundamental forms. The trace-free second fundamental form is the first example, typically denoted $\IIo := \iota^* \nabla n$, where $n$ is any extension of the unit conormal away from $\Sigma$. The third conformal fundamental form has formula $\IIIo := \iota^* W(\hat{n}, \cdot, \cdot, \hat{n})$, where $W$ is the Weyl tensor. The fourth fundamental form has a formula (except in $d=5$) given in abstract index notation as
$$\IVo_{ab} := \iota^* (n^c C_{c(ab)}) + H \IIIo_{ab} + \tfrac{1}{d-5} \bar{\nabla}^c \iota^* (W_{c(ab)d} n^d)\,,$$
where $C_{abc} := \tfrac{1}{d-3} \nabla^d W_{dcab}$ is the Cotton--York tensor, $H$ is the mean curvature, and round brackets $()$ denote unit-normalized symmetrization.

We are now equipped to prove the main result.


\section{Main Result}
To prove the main result, we first need a technical non-existence result.
We define a map sending a conformal hypersurface embedding to the space of all natural conformal hypersurface differential operators on that embedding with the relevant domain and codomain and a specific transverse order:
\begin{align*}
\mathcal{N}^{(k)}_{w,w'} : &\mathcal{M}^d \rightarrow\;  2^{\operatorname{Ops}^{(d)}} \\
&(\Sigma,(M,\cc),\iota) \mapsto \; \{N^{(k)}_{w,w'} \in \operatorname{Ops}^{(d)} \; | \\& \;  N^{(k)}_{w,w'} : \Gamma(\odot^2_\circ T^* M[w]) \rightarrow \Gamma(\odot^2_{\circ} T^* \Sigma[w'])\; |  \; \operatorname{to}(N^{(k)}_{w,w'}) = k\}\,.
\end{align*}
In the above, the \textit{transverse order} $\text{to}(op)$ of a differential operator $op: \Gamma(T^\Phi M[w]) \rightarrow \Gamma(T^\Psi \Sigma[w'])$ is defined as the integer $k$ such that for every $\ell > k$ we have that, in any trivialization $g \in \cc$, there exists a section $h \in \Gamma(T^\Phi M)$ such that
$$op(s^{\ell} h) = 0 \qquad \text{ and } \qquad op(s^k h) \neq 0\,,$$
for the unit defining function $s$ of $\Sigma$.

Our first aim, then, is to show that for certain values of $d$, $k$, $w$, and $w'$, $\mathcal{N}^{(k)}_{w,w'}(m) = \emptyset$ for every $m \in \mathcal{M}^d$.

\begin{lemma} \label{operator-nonexistence}
Let $d \geq 5$ be an odd integer and let $m \in \{2, \ldots, \tfrac{d-1}{2}\}$. Then for every $m \in \mathcal{M}^d$,
$$\mathcal{N}_{3-m,m-d+2}^{(d-2m+1)}(m) = \emptyset\,.$$
\end{lemma}
\begin{proof}
Give any embedding $(\Sigma,(M^d,\cc),\iota)$, we must show that there exists no differential operator
$$N^{(d-2m+1)}_{3-m,m-d+2} : \Gamma(\odot^2_\circ T^* M[3-m]) \rightarrow \Gamma(\odot^2_{\circ} T^* \Sigma[m-d+2])$$
with transverse order $d-2m+1$. Without loss of generality, such an operator may be written (up to an overall nonzero multiplicative constant) in a choice of scale $g \in \cc$ the form
$$N^{(d-2m+1)}_{3-m,m-d+2} = \otop \nabla_n^{d-2m+1} + \text{ltots}\,,$$
where $\text{ltots}$ indicates (possibly differential) operators with transverse order strictly smaller than $d-2m+1$. We now consider conformal transformations of the metric, $g \mapsto \Omega^2 g$. For some tensor $t_{ab} \in \Gamma(\odot^2_\circ T^* M)$ with conformal transformation given by 
$$t_{ab} \mapsto \Omega^{w} (t_{ab} + r_{ab})\,,$$
we compute
\begin{align*}
\nabla_n t_{ab} \mapsto \Omega^{w-1}(n^c + s \Upsilon^c)(&\nabla_c t_{ab} + (w-2) \Upsilon_c t_{ab} - \Upsilon_a t_{cb} - \Upsilon_b t_{ac} \\&+ \Upsilon^d t_{da} g_{bc} + \Upsilon^d t_{db} g_{ca}  + \mathcal{O}(r)) \\
=\Omega^{w-1}(\nabla_n t_{ab} + s &\nabla_{\Upsilon} t_{ab} + \mathcal{O}(n \cdot \Upsilon)+ \mathcal{O}(n \cdot t) + \mathcal{O}(\Upsilon \cdot t) \\&+ \mathcal{O}(|\Upsilon|^2) + \mathcal{O}(r))\,.
\end{align*}
where $\mathcal{O}(x)$ denotes any tensor structure that depends on at least one power of $x$ or its derivatives,  $\Upsilon := d \log \Omega$, and $s$ is a unit defining function for $\Sigma \hookrightarrow (M^d,g)$. As $\nabla_n s \= 1$, we find that we may describe the higher-order transformation according to
\begin{align*}
\otop \nabla_n^{d-2m+1} t_{ab} \stackrel{\Sigma}{\mapsto} \bar{\Omega}^{w-d+2m-1} &(\nabla_n^{d-2m+1} t_{ab} + (d-2m) \bar{\nabla}_{\bar{\Upsilon}}  (\nabla_n^{d-2m-1} t_{ab})^\top + \text{lots}(t_{ab}) \\&+ \mathcal{O}(n \cdot \Upsilon) + \mathcal{O}(n \cdot t) + \mathcal{O}(\Upsilon \cdot t) + \mathcal{O}(|\Upsilon|^2) + \mathcal{O}(r)) ^{\otop} \\
\= \bar{\Omega}^{w-d+2m-1}&(\nabla_n^{d-2m+1} t_{ab} + (d-2m) \bar{\nabla}_{\bar{\Upsilon}} (\Delta^{\tfrac{d-1}{2} - m} t_{ab})^\top \\&- (d-2m) \bar{\nabla}_{\bar{\Upsilon}} \bar{\Delta}^{\frac{d-1}{2}-m} t_{ab}^\top + \text{lots}(t_{ab}) \\&+ \mathcal{O}(n \cdot \Upsilon) + \mathcal{O}(n \cdot t) + \mathcal{O}(\Upsilon \cdot t) + \mathcal{O}(|\Upsilon|^2) + \mathcal{O}(r))^{\otop}
\end{align*}
where $\text{lots}(t_{ab})$ stands for any linear differential operator on $t$ that has strictly fewer than $d-2m$ derivatives on $t$. The second identity follows from the first from the simple observation that $\nabla_n^2 = \Delta - \bar{\Delta} \circ \top$  up to lower order derivatives. Crirtically, any such conformally invariant differential operator must include in its universal symbolic formula terms that, after simplification, transform into $\bar{\nabla}_{\bar{\Upsilon}} \circ \top \circ \Delta^{\tfrac{d-1}{2} -m}$ and $\bar{\nabla}_{\bar{\Upsilon}} \bar{\Delta}^{\tfrac{d-1}{2} - m}$.

To see what terms can possibly produce such counterterms, consider the conformal transformation of $\bar{\Delta} t_{ab}^\top$:
\begin{align*}
\bar{\Delta} t_{ab}^{\top} \stackrel{\Sigma}{\mapsto} \bar{\Omega}^{w-2} (&\bar{\Delta} t_{ab}^\top + (d+2w-7) \bar{\nabla}_{\bar{\Upsilon}} t_{ab}^\top \\&+ \mathcal{O}(\bar{\nabla} \bar{\Upsilon}) + \mathcal{O}(\Upsilon \cdot t) + \mathcal{O}(\bar{\nabla} \cdot t^\top) + \mathcal{O}(|\Upsilon|^2) + \mathcal{O}(n \cdot \Upsilon) + \mathcal{O}(n \cdot t) + \mathcal{O}(r)\,.
\end{align*}
Thus, we can see that adding a term of the form $\bar{\Delta}(\Delta^{\tfrac{d-1}{2}-m} t_{ab})^\top$ is sufficient to cancel out transformations that result in $\bar{\nabla}_{\bar{\Upsilon}} (\Delta^{\tfrac{d-1}{2} -m} t_{ab})^\top$, at least at leading order---subsequent lower order terms, and terms with other tensor structures, may be required to cancel out terms that have been ignored here.

But it is straightforward to check that there are no other operators with the correct homogeneity and derivative order aside from $\bar{\Delta}$ that transform this way. Thus, the unique counterterm that can cancel $\bar{\nabla}_{\bar{\Upsilon}} \bar{\Delta}^{\tfrac{d-1}{2} -m} t^\top_{ab}$ takes the form $\bar{\Delta}^{\tfrac{d+1}{2}-m} t^\top_{ab}$. However, upon iterating the above transformation of $\bar{\Delta} t^{\top}_{ab}$, we find that
\begin{align*}
\otop \bar{\Delta}^{\tfrac{d+1}{2}-m} t^\top_{ab} \stackrel{\Sigma}{\mapsto} \Omega^{w-d+2m-1}(&\bar{\Delta}^{\tfrac{d+1}{2}-m} t^\top_{ab} + (\tfrac{d+1}{2}-m)(2m+2w-6) \bar{\nabla}_{\bar{\Upsilon}} \bar{\Delta}^{\tfrac{d-1}{2} -m} t^{\top}_{ab} \\&+
\mathcal{O}(\bar{\nabla} \bar{\Upsilon}) + \mathcal{O}(\Upsilon \cdot t) + \mathcal{O}(\bar{\nabla} \cdot t^\top) + \mathcal{O}(|\Upsilon|^2) \\&+ \mathcal{O}(n \cdot \Upsilon) + \mathcal{O}(n \cdot t) + \mathcal{O}(r))^{\otop}\,.
\end{align*}
However, when $w = 3-m$, the required counterterm vanishes. Thus, we have established that there is no natural hypersurface differential operator that may be added to $\otop \nabla_n^{d-2m+1}$ such that, on weight $3-m$ tensors of the specified type, the conformal transformation vanishes entirely. Hence, no such natural conformal hypersurface differential operators exist.
\end{proof}


Using this lemma, we may now prove the main theorem.
\begin{proof}[Proof of~\ref{main-theorem}]
When $d \geq 4$ is even, such conformal fundamental forms have already been constructed; see~\cite{Blitz1} for the construction up to order $d-1$.

When $d \geq 5$ is odd, the construction of each conformal fundamental form up to order $\tfrac{d+1}{2}$ was given in~\cite{BGW1}. It suffices, then, to establish that there does not exist a section of $\mathcal{M}^d \rightarrow E_{d-k+1}$ that is a right-inverse of the natural projection and has transverse order $d-k$ for each $k \in \{2, \ldots, \tfrac{d-1}{2}\}$.

Aiming at a contradiction, we consider any embedding $(\Sigma,(M^d,\cc),\iota)$ and suppose that there exists such a section, call it $\FF{d-k+1}$. Then, for a given choice of trivialization $g \in \cc$ and a choice of orthonormal coordinates in a small open neighborhood of a point $p \in \Sigma$, call them $(s, x^1, \cdots, x^{d-1})$ with $s$ a defining function for $\Sigma$ and $|ds|_g =1$, it follows by definition that
$$\FF{d-k+1}_{ab} \= \otop \partial_s^{d-k} g_{ab} + \text{ltots}(g)\,.$$

On the other hand, consider the normal derivative operator $\ID_{\sigma} : \Gamma(\odot^2_{\circ} T^* M[w]) \rightarrow \Gamma(\odot^2_{\circ} T^*M[w-1])$ defined in ~\cite{BGW1} and the conformal extension of the second fundamental form $\IIo^{\rm e}_{ab} := \nabla_{(a} n_{b)_{\circ}} +\mathring{P}_{ab} s$. From the proof of~\cite[Lemma 4.7]{BGW1}, we see that the tensor
$$\FF{k}^{\rm e} := \ID_{\sigma}^{k-2} \IIo^{\rm e}$$
is well-defined, and upon pull back to $\Sigma$, is a non-zero multiple of $\FF{k}$. As such, we may see that that
$$\FF{k}^{\rm e} = \alpha \operatorname{tf} \circ \, \partial_{s}^{k-1} g + \text{ltots}(g) + \mathcal{O}(s)\,,$$
where $\alpha$ is a non-zero constant. Observe that the term $\mathcal{O}(s)$ may in principle depend on higher-order derivatives of the metric, and poses a risk upon transverse differentiation of cancelling the term proportional to $\alpha$. However, as $\otop \nabla_n^{d-2k+1} \FF{k}^{\rm e}$ captures the same leading transverse order structure as $\otop \nabla_n^{d-k-3} \IIo^{\rm e}$ by construction, it follows that
$$\FF{d-k-1}_{ab} \= (\beta\otop \nabla_n^{d-2k+1} + \text{ltots})\,\FF{k}^{\rm e}_{ab} \,,$$
where $\beta$ is a non-zero constant.

Now observe that as the leading order term in $\FF{k}^{\rm e}$ is proportional to $\operatorname{tf} \circ \, \partial_s^{k-1} g$, it follows that, by considering those embeddings $\iota_h : \Sigma \hookrightarrow (M,g+s^{k-1}h)$ with $h \in \Gamma(\odot^2 T^* M)$, the map $\FF{k}^{\rm e} : \mathcal{M}^d \rightarrow \Gamma(\odot^2_{\circ} T^* M[3-k])$ is surjective. But then, there exists some operator
$$\beta \otop \nabla_n^{d-2k+1} + \text{ltots} : \Gamma(\odot^2_\circ T^* M[3-k]) \rightarrow \Gamma(\odot^2_{\circ} T^* \Sigma[2-d+k])\,.$$

But from Lemma~\ref{operator-nonexistence}, we have established that no such operator exists, as
$$\mathcal{N}_{3-k,k-d+2}^{(d-2k+1)}(\Sigma,(M,\cc),\iota) = \emptyset\,.$$
Thus we have a contradiction, completing the proof.
\end{proof}

\begin{remark}
In~\cite{BGW1}, it was shown that there exist \textit{conditional} conformal fundamental forms that appear to contradict Theorem~\ref{main-theorem}. These are sections of $\odot^2_{\circ} T^* \Sigma$ that have all of the features of conformal fundamental forms, except they are only conformally-invariant when certain lower-order conformal fundamental forms vanish. However, as defined in this work, conformal fundamental forms are functionals of conformal hypersurface embeddings rather than specific sections on specific embeddings. Hence, the non-existence proof in this work merely shows that there are no \textit{non-conditional} high-order conformal fundamental forms in odd dimensions.
\end{remark}

\section*{Acknowledgements}
The author would like to thank Jeffrey Case and A. Rod Gover for enlightening conversations that informed on the creation of this work during DGA 2025.

\bibliographystyle{plain}
\bibliography{biblio}

\end{document}